\documentclass{amsart}
\usepackage[english]{babel}
\usepackage{amsfonts}
\usepackage{amssymb}
\usepackage{amsmath,mathrsfs}
\usepackage{amsthm}
\usepackage{graphicx}

\usepackage{todonotes}

\def\bN {\mathbf{N}}

\def\bR {\mathbf{R}}
\def\bS {\mathbf{S}}

\def\bZ {\mathbf{Z}}

\def\fH {\mathfrak{H}}

\def\cF {\mathcal{F}}
\def\cG {\mathcal{G}}

\def\cQ {\mathcal{Q}}

\def\cX {\mathcal{X}}

\def\la {\langle}
\def\ra {\rangle}

\def\rstr{\big|}

\newcommand{\Div}{\operatorname{div}}

\newcommand{\Diam}{\operatorname{diam}}
\newcommand{\Diff}{\operatorname{Diff}}

\newcommand{\Det}{\operatorname{det}}
\newcommand{\Tr}{\operatorname{trace}}

\newcommand{\Dist}{\operatorname{dist}}

\newcommand{\Id}{\operatorname{Id}}
\newcommand{\id}{\operatorname{id}}

\newcommand{\ba}{\begin{aligned}}
\newcommand{\ea}{\end{aligned}}

\newcommand{\be}{\begin{equation}}
\newcommand{\ee}{\end{equation}}

\def\scrL{\mathscr{L}}

\def\a {{\alpha}}
\def\b {{\beta}}
\def\g {{\gamma}}
\def\Ga {{\Gamma}}
\def\de {{\delta}}
\def\eps {{\epsilon}}
\def\th {{\theta}}

\def\l {{\lambda}}
\def\L {{\Lambda}}

\def\om {{\omega}}
\def\Om {{\Omega}}

\def\d {{\partial}}
\def\grad {{\nabla}}

\newcommand{\mc}{\mathcal}
\newcommand{\pt}{\partial}

\newcommand{\br}{\mathbf{R}}

\renewcommand{\om}{\omega}
\newcommand{\te}{\theta}

\renewcommand{\eps}{\varepsilon}

\renewcommand{\r}{\rho}

\renewcommand{\[}{\left[}
\renewcommand{\]}{\right]}

\newtheorem{thm}{Theorem}
\newtheorem{lem}[thm]{Lemma}

\def\be{\begin{equation}}
\def\ee{\end{equation}}
\def\bea{\begin{eqnarray}}
\def\eea{\end{eqnarray}}
\def\lb{\label}
\def\wto{\rightharpoonup}

\numberwithin{thm}{section}
\numberwithin{equation}{section}
\numberwithin{figure}{section}

\begin{document}

\title[Quantization of 2D measures]{Quantization of measures and gradient flows: \\a perturbative approach in the \\$2$-dimensional case.}

\author[E. Caglioti]{Emanuele Caglioti}
\address[E.C.]{Sapienza Universit\`a di Roma, Dipartimento di Matematica Guido Castelnuovo, Piazzale Aldo Moro 5, 00185 Roma, Italy}
\email{francois.golse@polytechnique.edu}

\author[F. Golse]{Fran\c cois Golse}
\address[F.G.]{CMLS, \'Ecole polytechnique, CNRS, Universit\'e Paris-Saclay , 91128 Palaiseau Cedex, France}
\email{caglioti@mat.uniroma1.it}

\author[M. Iacobelli]{Mikaela Iacobelli}
\address[M.I.]{University of Cambridge, DPMMS Centre for Mathematical Sciences, Wilberforce road, Cambridge CB3 0WB, United Kingdom}
\email{iacobelli@maths.cam.ac.uk}

\begin{abstract}
In this paper we study a perturbative approach to the problem of quantization of measures in the plane. Motivated by the fact that, as the number of points tends to infinity, hexagonal lattices are asymptotically optimal from an energetic point 
of view \cite{FT,Gr,M}, we consider configurations that are small perturbations of the hexagonal lattice and we show that: (1) in the limit as the number of points tends to infinity, the hexagonal lattice is a strict minimizer of the energy; (2) the 
gradient flow of the limiting functional allows us to evolve any perturbed configuration to the optimal one exponentially fast. In particular, our analysis provides a solid mathematical justification of the asymptotic optimality of the hexagonal 
lattice among its nearby configurations.
\end{abstract}

\date{\today}

\maketitle




\section{Introduction}


The term \emph{quantization} refers to the process of finding an \emph{optimal} approximation of a $d$-dimensional probability density by a convex combination of a finite number $N$ of Dirac masses. The quality of such an approximation is 
measured in terms of the Monge-Kantorovich or Wasserstein metric.

The need of such approximations first arose in the context of information theory in the early 1950s. The idea was to see the quantized measure as the digitalization of an analog signal which should be stored on a data storage medium or transmitted 
via a channel \cite{BW, GG}. Another classical application of the quantization problem concerns numerical integration, where integrals with respect to certain probability measures needs to be replaced by integrals with respect to a good discrete 
approximation of the original measure \cite{GPP}. For instance, quasi-Monte Carlo methods use low discrepancy sequences, and the notion of discrepancy can be regarded as one approach to the quantization problem. See \cite{Caflisch} for an
introduction to this subject, especially section 5 for a presentation of the notion of discrepancy of a sequence, and section 7 for applications in the context of rarefied gas dynamics. Moreover, this problem has applications in cluster analysis, pattern 
recognition, speech recognition, stochastic processes (sampling design) and mathematical models in economics (optimal location of service centers). For a detailed exposition and a complete list of references see the monograph \cite{CGI}. 

We now introduce the theoretical setup of the problem. Given $r\ge 1$, consider $\r$ a probability density on an open set $\Omega \subset \br^d$ with finite $r$-th order moment,
$$
\int_{\Omega}|y|^r \rho(y)dy<\infty.
$$
Given $N$ points $x^{1}, \ldots, x^{N} \in \Omega,$ we seek the best approximation of $\rho,$ in the sense of Monge-Kantorovich, by a convex combination of Dirac masses centered at $x^{1}, \ldots, x^{N}.$ Hence one minimizes
$$
\inf \bigg\{ MK_r\bigg(\sum_im_i \delta_{x^i}, \r(y)dy\bigg)\,:\, m_1, \ldots, m_N\ge0, \  \sum_{i=1}^Nm_i=1\bigg\},
$$
with
$$
MK_r(\mu, \nu):=\inf\bigg\{\int_{\Omega\times\Omega}|x-y|^rd\gamma(x,y)\,:\, \pi \in \Pi(\mu,\nu)\bigg\},
$$
where $\Pi(\mu,\nu)$ is the set of all Borel probability measures on $\Omega \times \Omega$ whose marginals  onto the first and second component are given by $\mu$ and $\nu$ respectively. In other words, a Borel probability $\pi$ measure on 
$\Omega \times \Omega$ belongs to $\Pi(\mu,\nu)$ if
$$
\iint_{\Omega\times\Omega}(\phi(x)+\psi(y))\pi(dxdy)=\int_\Omega\phi(x)\mu(dx)+\int_\Omega\psi(y)\nu(dy)
$$
for each $\phi,\psi\in C_c(\Omega)$ (see \cite{AGS,Vil03} for more details on the Monge-Kantorovitch distance between probability measures).

As shown in \cite{GL}, the following facts hold:
\begin{enumerate}
\item
The best choice of the masses $m_i$ is given by
$$
m_i:= \int_{V(x^i|\{x^{1}, \ldots, x^{N}\} )} \r(y)dy,
$$
where 
$$
V(x^i|\{x^{1}, \ldots, x^{N}\}):= \{y \in \Omega\ : \ |y-x^i| \le |y-x^{j}|,\  j \in 1, \ldots, N \}
$$ 
is the so called  \emph{Voronoi cell} of $x^i$ in the set $x^1, \ldots, x^N.$

\item The following identity holds:
\begin{multline*}
\inf \bigg\{ MK_r\bigg(\sum_im_i \delta_{x^i}, \r(y)dy\bigg)\,:\, m_1, \ldots, m_N\ge0, \  \sum_{i=1}^N m_i=1\bigg\}
\\
=\cQ_{N,r} (x^{1}, \ldots, x^{N}) ,
\end{multline*}
where
 $$
\cQ_{N,r} (x^{1}, \ldots, x^{N}) := \int_\Omega \underset{1\le i \le N}{\mbox{min}} | x^i-y |^r\rho(y)dy.
$$
\end{enumerate}

Assume that the points $x^{1}, \ldots, x^{N}$ are chosen in an optimal way so as to minimize the functional $\cQ_{N,r}: (\br^d)^N \to \br^+$; then in the limit as $N$ tends to infinity these points distribute themselves accordingly to a probability 
density proportional to $\rho^{d/{d+r}}.$ In other words, by \cite[Chapter 2, Theorem 7.5]{GL} one has
\be\label{flow1:close}
\frac{1}{N}\sum_{i=1}^N \delta_{x^i}\wto\frac{\rho^{d/{d+r}}(x)dx}{\displaystyle\int_\Omega \rho^{d/{d+r}}(y)dy}
\ee
weakly in the sense of Borel probability measures on $\Omega$ as $N\to\infty$.

These issues are relatively well understood from the point of view of the calculus of variations \cite[Chapter 1, Chapter 2]{GL}. Moreover, in \cite{CGI} we considered a gradient flow approach to this problem in dimension $1$. Now we will explain 
the heuristic of the dynamical approach and the main difficulties in extending our result to higher dimensions.


\subsection{A dynamical approach to the quantization problem}


Given $N$ points $x^{1}_0, \ldots, x^{N}_0$, we consider their evolution  under the gradient flow generated by $\cQ_{N,r}$, that is, we solve the system of ODEs in $(\br^d)^N$
\be
\label{flow1:eq:ODEintro}
\left\{
\begin{array}{rl}
\bigl(\,\dot x^{1}(t),\ldots,\dot x^{N}(t)\,\bigr)&=-\nabla \cQ_{N,r}\bigl(x^{1}(t),\ldots,x^{N}(t)\bigr),\\
\bigl(x^{1}(0),\ldots,x^{N}(0)\bigr)&=(x^{1}_0, \ldots, x^{N}_0).
\end{array}
\right.
\ee

As usual in gradient flow theory, as $t$ tends to infinity one expects that the points $\bigl(x^{1}(t),\ldots,x^{N}(t)\bigr)$ converge to a minimizer $(\bar x^{1},\ldots,\bar x^N)$ of $\cQ_{N,r}.$ Hence, in view of \eqref{flow1:close}, the empirical measure 
$$
\frac{1}{N}\sum_{i=1}^N \delta_{\bar x^i}
$$
is expected to converge weakly in the sense of probability measures to 
$$
\frac{\rho^{d/{d+r}}}{\displaystyle\int_\Omega \rho^{d/{d+r}}(y)dy}dx
$$ 
as $N \to \infty$.

Our approach of this problem involves exchanging the limits as $t\to \infty$ and $N\to\infty$. More precisely, we first pass to the limit in the ODE above as $N\to\infty$, and take the limit in the resulting PDE as $t\to+\infty$. For this, we take a set of 
reference points $(\hat x^{1},\ldots,\hat x^N)$ and we parameterize a general family of $N$ points $x^i$ as the image of $\hat x^i$ via a slowly varying smooth map $X:\br^d\to \br^d$, that is
 $$
 x^i=X(\hat x^i).
 $$
In this way, the functional $\cQ_{N,r}(x^{1},\ldots,x^{N})$ can be rewritten in terms of the map $X$ and (a suitable renormalization of it) should converge to a functional $\mathcal F[X]$. Hence, we can expect that the evolution of $x^i(t)$ for $N$ large
is well-approximated by the $L^2$-gradient flow of $\mathcal F$.

Although this formal argument may sound convincing, already the 1-dimensional case is rather delicate. We briefly review the results of \cite{CGI} below.


\subsection{The 1D case}


Without loss of generality let $\Omega$ be the open interval $(0,1)$, and consider $\r$ a smooth probability density on $\Omega.$ In order to obtain a continuous version of the functional
$$
\cQ_{N,r}(x^{1}, \ldots, x^{N})=\int_{0}^1\underset{1\le i \le N}{\mbox{min}} | x^i-y |^r\rho(y)\,dy,
$$ 
with $0\le x^{1}\le \ldots \le x^{N}\leq 1$, assume that
$$
x^i=X\bigg(\frac{i-1/2}{N}\bigg), \qquad i=1, \ldots, N
$$
with $X: [0,1] \to [0,1]$ a smooth non-decreasing map such that $X(0)=0$ and $X(1)=1$. Then,
$$
N^r\cQ_{N,r}(x^{1}, \ldots, x^{N}) \longrightarrow C_r\int_0^1 \r(X(\te))|\pt_\te X(\te)|^{r+1}d\te:=\mc{F}[X]
$$
as $N\to \infty,$ where $C_r:=\frac{1}{2^r(r+1)}.$

By a standard computation, we obtain the gradient flow PDE for $\mc{F}$ for the $L^2$-metric,
\begin{multline}\label{flow1:gradient flow}
\pt_tX(t,\te)=C_r\Big((r+1)\pt_\te\big(\r(X(t,\te))|\pt_\te X(t,\te)|^{r-1}\pt_\te X(t,\te)\big)
\\
-\r'(X(t,\te))|\pt_\te X(t,\te)|^{r+1} \Big),
\end{multline}
coupled with the Dirichlet boundary condition
\begin{equation}
\label{flow1:eq:boundary}
X(t,0)=0, \qquad X(t,1)=1.
\end{equation}

Our main result in \cite{CGI} shows that, provided that $r=2,$ that $\|\rho-1\|_{W^{2,\infty}(0,1)} \ll 1,$ and  that the initial datum is smooth and increasing, the discrete and the continuous gradient flows remain {\it uniformly} close in $L^2$ for {\it all} 
times. In addition, by entropy-dissipation inequalities for the PDE, we show that the continuous gradient flow converge exponentially fast to the stationary state for the PDE, which is seen in Eulerian variables to correspond to the measure 
$$
\frac{\rho^{1/3}(x)dx}{\displaystyle\int_0^1\rho^{1/3}(y)dy}
$$
as predicted by \eqref{flow1:close}.


\subsection{The 2D case: setting and main result}


Our goal is to extend the result above to higher dimensions. As a first step, it is natural to consider the quantization problem for the uniform measure in space dimension $2$. The main advantage in this situation is that optimal configurations are 
known to be asymptotically hexagonal lattices \cite{FT,Gr,M}. (Notice however that the reference \cite{M} considers the $2$-dimensional quantization problem in the Monge-Kantorovich distance of exponent $1$, i.e. with $r=1$, at variance with 
our approach in the present paper which assumes $r=2$.) Hence, it will be natural to use the vertices of the optimal, hexagonal lattice as reference points $\hat x^i$, and to assume that the time-dependent configuration of points are obtained as 
slowly varying deformations of the optimal configuration.

More precisely, we shall consider the following setting. Let us consider a regular hexagonal tessellation of the Euclidean plane $\bR^2$. Up to some inessential displacement, one can choose the centers of the hexagons to be the vertices of the 
regular lattice
$$
\scrL:=\bZ e_1\oplus\bZ e_2\qquad\hbox{ where }e_1=(1,0)\hbox{ and }e_2=\left(\tfrac12,\tfrac{\sqrt{3}}2\right).
$$ 
Let $\Pi$ be the fundamental domain of $\bR^2/\scrL$ centered at the origin defined as follows:
$$
\Pi:=\{x_1e_1+x_2e_2\,:\,|x_1|,|x_2|\le\tfrac12\}\,.
$$
Henceforth, we consider the sequence of scaled lattices homothetic to $\scrL$, of the form $\eps\scrL$ with $\eps=1/n$ and $n\in\bN^*$. 

The slowly varying deformations of $\scrL$ used in this work are discrete sets of the form $\cX_\eps:=X(\eps\scrL)$ with $\eps=1/n$ and $n\in\bN^*$, where $X\in\Diff^1(\mathbf{R}^2)$, i.e. $X$ is a $C^1$-diffeomorphism of $\bR^2$ onto itself.
We shall assume that $X$ satisfies the following properties:

\smallskip
\noindent
(a) $X$ is a periodic perturbation of the identity map, i.e.
$$
X(x+l)=X(x)+l\quad\hbox{ for all }x\in\bR^2\hbox{  and }l\in\scrL\,;
$$
(b) $X$ is $C^1$-close to the identity map, i.e.,
$$
\|X-\mbox{id}\|_{W^{1,\infty}}<\eta\quad\hbox{where }\eta\ll 1\,.
$$ 
(c) $X$ is centered at the origin, i.e.
$$
\int_\Pi X(x)dx=0\,.
$$
(This last condition does not restrict the generality of our approach: if $X$ fails to satisfy property (c), set 
$$
\la X\ra:=\int_\Pi X(x)dx=\int_\Pi(X(x)-x)dx\,;
$$
then $\la X\ra<\eta$ and $\hat X(x)=X(x)-\la X\ra$ satisfies properties (a) and (c), and property (b) where $\eta$ is replaced by $2\eta$.)

\smallskip
If $\eta$ is sufficiently small, then for every $x\in\cX_\eps$ the Voronoi cell $V(x|\cX_\eps)$ centered in $x$ with respect to the set of points $\cX_\eps$ is an hexagon. Indeed, the angle between two adjacent edges starting from any vertex in the 
deformed configuration $\cX_\eps$ is $\pi/3+O(\eta)$ by the mean value theorem. Hence the center of the circumscribed cercle to any triangle with nearest neighbor vertices in the deformed configuration lies in the interior of this triangle, and 
ethe perpendicular bissectors of the edges of the triangle intersect at the center of the circumscribed circle with an angle $2\pi/3+O(\eta)$. Hence the family of all such centers are the vertices of an (irregular) hexagonal tessellation of the plane, 
which is the Voronoi tessellation of the deformed configuration.

\medskip
To avoid all difficulties pertaining to boundary conditions, we formulate our quantization problem in the ergodic setting. In other words, we consider the discrete quantization functional averaged over disks with radius $L\gg 1$:
$$
\cG_{\eps,L}(\cX_\eps):=\int_{B(0,L)}\Dist(y,\cX_\eps)^2dy\,.
$$
Our first main result describes the asymptotic behavior of $\cG_{\eps,L}$ as follows.

\begin{thm}\lb{T-LimG}
Assume that $X\in\Diff^1(\mathbf{R}^2)$ satisfies the properties (a-c) above, and that $\eta$ is small enough. For each $\eps=1/n$ with $n\in\bN^*$
$$
\frac1{\pi L^2}\cG_{\eps,L}(\cX_\eps)\to\int_{X(\Pi)}\Dist(y,\cX_\eps)^2dy\quad\hbox{ as }L\to\infty\,.
$$
Moreover
$$
\int_{X(\Pi)}\Dist(y,\cX_\eps)^2dy\sim\eps^2\cF(X)\quad\hbox{ as }\eps\to 0,
$$
where $\cF$ is given by
$$
\cF(X):=\int_\Pi F(\grad X(x))dx\,.
$$
In this expression, the function $F$ is defined by the formula
$$
F(M):=\tfrac1{48}\sum_{\om\in\{e_1,e_2,e_1-e_2\}}|M\om|^4\Phi(\om,M)(3+\Phi(\om,M)^2)\,,
$$
where
$$
\Phi(e,M):=\sqrt{\frac{|MRe|^2|MR^Te|^2}{\tfrac34|\Det(M)|^2}-1}
$$
for each invertible $2\times 2$-matrix $M$ with real entries and each unit vector $e$, and where $R$ designates the rotation of an angle $\tfrac{\pi}3$ centered at the origin.
\end{thm}

\smallskip
The proof of Theorem \ref{T-LimG} occupies sections \ref{S2} and \ref{S3} below.

\smallskip
Observe that the integrand in the definition of $\cF$ depends exclusively on the gradient of the map $\grad X$. In particular, the integrand in $\cF$ involves the Jacobian determinant $\Det(\grad X)$ of the deformation map. However
the dependence of $\cF$ on $\Det(\grad X)$ becomes singular as $\Det(\grad X)\to 0$. Therefore, our analysis is restricted to small perturbations of the uniform, hexagonal tessellation of the plane. This is obviously consistent with the 
fact that we postulated that our configuration of points remains close to the uniform hexagonal tessellation in order to arrive at the explicit expression of $\cF$ given in Theorem \ref{T-LimG}. 

\smallskip
On the other hand, at variance with the 1D case, the limiting function $\cF$ does not depend on $X$ only through its Jacobian determinant. This seriously complicates the Eulerian formulation of the 2D case, which was relatively simple
in the 1D case, and which we used in a significant manner in our earlier work \cite{CGI}.

\smallskip
Our second result is a simplified expression for $F$ near the identity matrix.

\begin{thm}\lb{T-FnearI}
Let 
$$
S=\left(\begin{matrix} 1 & 0\\ 0&-1\end{matrix}\right)\,.
$$
There exists $0<\eta_0\ll 1$ such that, for all $M\in GL_2(\bR)$ such that $|M-I|\le\eta_0$, one has
$$
\ba
F(M)=&\tfrac1{16\sqrt{3}}\Det(M)\Tr(M^TM(2S-I))
\\
&+\tfrac1{64\sqrt{3}}\frac{\Tr(M^TM)^2\Tr(M^TMS)}{\Det(M)}
\\
&-\tfrac1{192\sqrt{3}}\frac{\Tr(M^TM)^3}{\Det(M)}-\tfrac1{48\sqrt{3}}\frac{\Tr(M^TMS)^3}{\Det(M)}\,.
\ea
$$
Moreover, for each $2\times 2$-matrix $N$ with real entries, one has
$$
\ba
48F(I+\eps N)=&\tfrac{10}{\sqrt{3}}+\tfrac{20}{\sqrt{3}}\eps\Tr(N)
\\
&+\tfrac1{\sqrt{3}}\eps^2\left(14\Det(N)+10\Tr(N)^2+3\Tr(N^TN)\right)+O(\eps^3)\,.
\ea
$$
\end{thm}

\smallskip
The proof of Theorem \ref{T-FnearI} is given in section \ref{S5} below.

\smallskip
Our third main result bears on the basic properties of the $L^2$-gradient flow of the asmptotic quantization functional $\cF$ obtained in Theorem \ref{T-LimG}.

\begin{thm}\lb{T-Gflow}
Let $X^{in}\in\Diff(\mathbf{R}^2)$ satisfy the properties (a-c) above, together with the condition
\be\lb{b'}
\|X^{in}-\id\|_{W^{\sigma,p}(\Pi)}\le\eps_0,
\ee
with $p>2$ and $1+2/p<\sigma$. Consider the PDE defining the $L^2$-gradient flow of $\cF$:
\be\lb{GflowF}
\left\{
\ba
{}&\d_tX_j(t,x)=-\frac{\de\cF}{\de X_j(t,x)}\,,\qquad j=1,2,
\\
&X(0,x)=X^{in}(x).
\ea
\right.
\ee
The $L^2$-gradient flow of $\cF$ starting from any initial diffeomorphism $X^{in}$ satisfying the conditions above exists and is unique. In other words, the Cauchy problem (\ref{GflowF}) has a unique solution $X$ defined for all $t>0$, and 
$X(t,\cdot)$ satisfies properties (a-c) for all $t>0$. Besides, the solution $t\mapsto X(t,\cdot)$ of (\ref{GflowF}) converges exponentially fast to the identity map as $t\to+\infty$: for each $X^{in}$ satisfying (a-c) and (\ref{b'}), there exist $C,\mu>0$ 
(depending on $\eps_0$) such that
$$
\|X(t,\cdot)-\id\|_{L^\infty(\Pi)}\le Ce^{-\mu t}
$$
for all $t>0$.
\end{thm}

\smallskip
Since $\cF$ depends on $\Det(\grad X)$, one cannot hope that $\cF$ has any convexity property. Morevoer, since the dependence of $\cF$ in $\Det(\grad X)$ is singular, one cannot hope that some compensations would offset the lack
of convexity coming from the determinant. For this reason, we consider initial configurations that are small perturbations of the hexagonal lattices, and we study in detail the linearization at the equilibrium configuration of the system of
equations defining the gradient flow of $\cF$. Combining this with some general $\eps$-regularity theorems for parabolic systems, we prove that the nonlinear evolution is governed by the linear dynamics, and in this way we can prove 
exponential convergence to the hexagonal (equilibrium) configuration. As we shall see, our proof of Theorem \ref{T-Gflow}, which occupies section \ref{S6} below, is based on tools coming from the regularity theory for parabolic systems,
 and this is why we need assumptions on the initial data in appropriate Sobolev spaces.

\smallskip
Moreover, our numerical simulations confirm the asymptotic optimality of the hexagonal lattice as the number of points tends to infinity --- see Figures \ref{F1_emanuele}, \ref{F2_emanuele}, and \ref{F3_emanuele}. The colored polygons 
in Figures \ref{F1_emanuele}, \ref{F2_emanuele}, and \ref{F3_emanuele} are the hexagons. Figure \ref{F3_emanuele} suggests that the minimizers may have some small $1$-dimensional defects with respect to the hexagonal lattice. 
This may be caused by the boundary conditions used in the the numerical simulation which are not periodic, at variance with the setting used in Theorems \ref{T-LimG} and \ref{T-Gflow}. Another possible explaination is that the particle 
system may remain frozen in some local minimum state. Also, the hexagonal tessellation is not the global minimizer for a finite number $N$ of points, and this is another difference between the discrete and the continuous problems.

\bigskip

\begin{figure}
\center
\includegraphics[width=6.5cm]{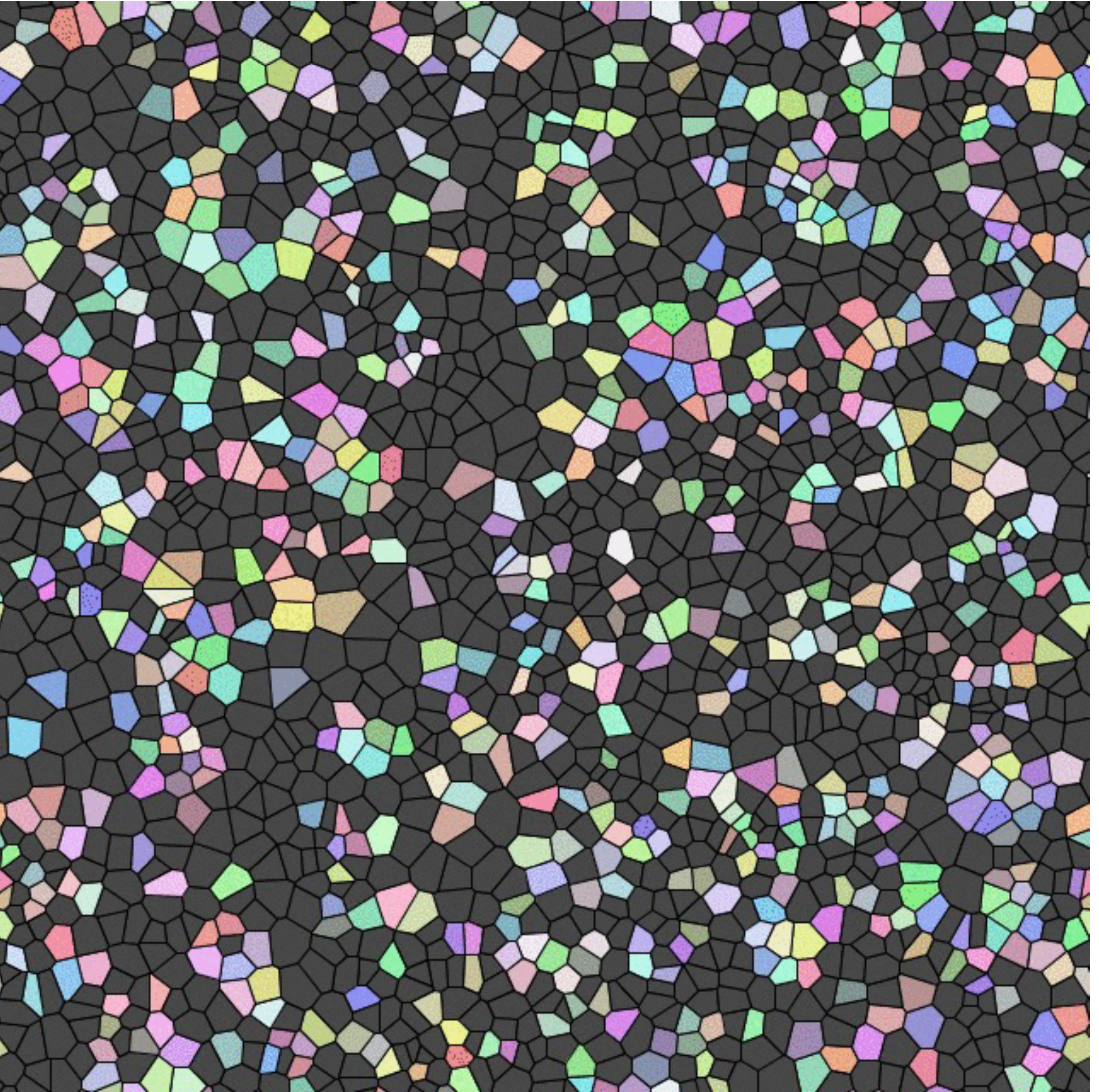}
\caption{720 points at time 0}
\label{F1_emanuele}
%
\center
\includegraphics[width=6.5cm]{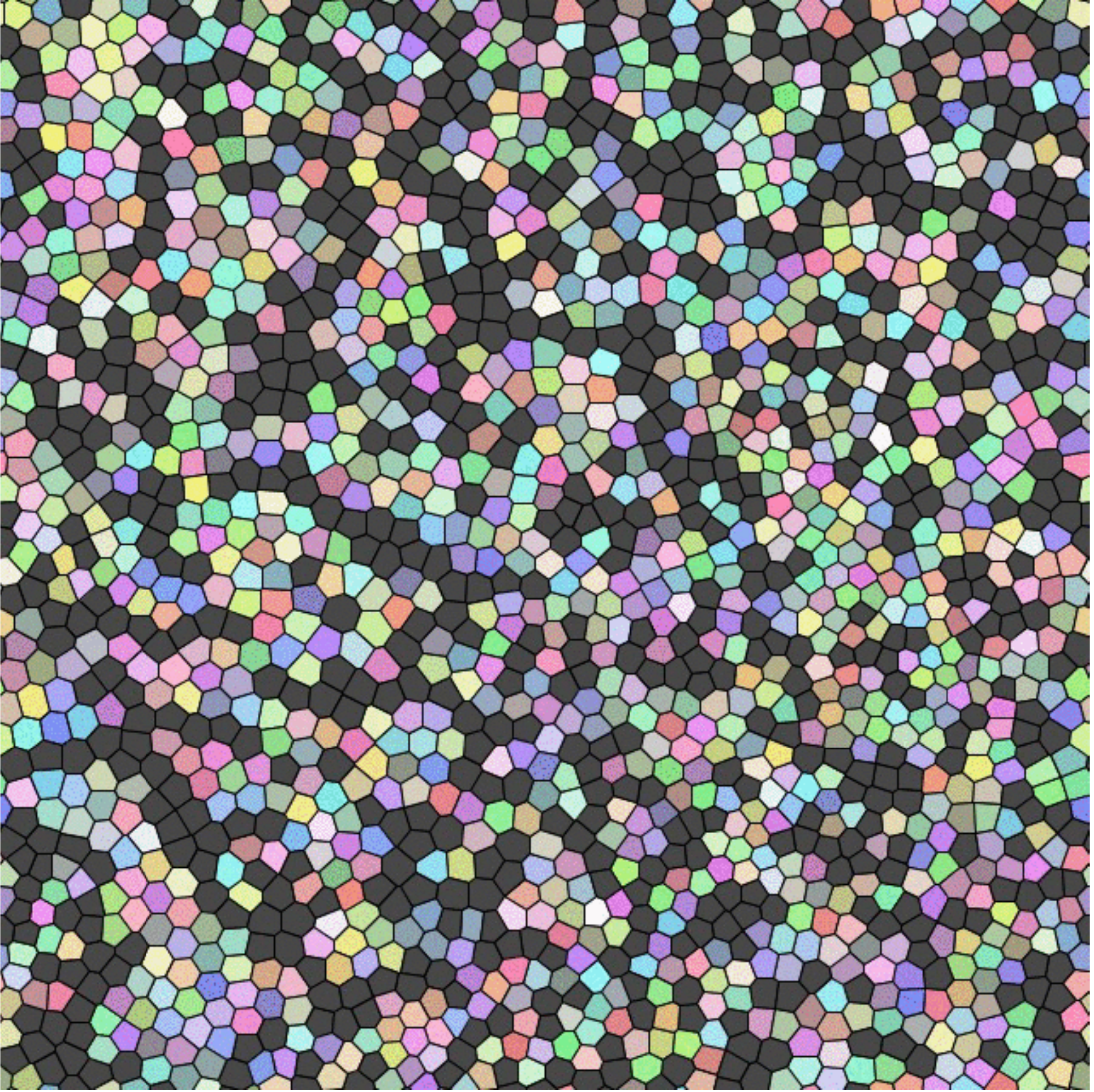}
\caption{720 points after 19 iterations}
\label{F2_emanuele}
%
\center
\includegraphics[width=6.5cm]{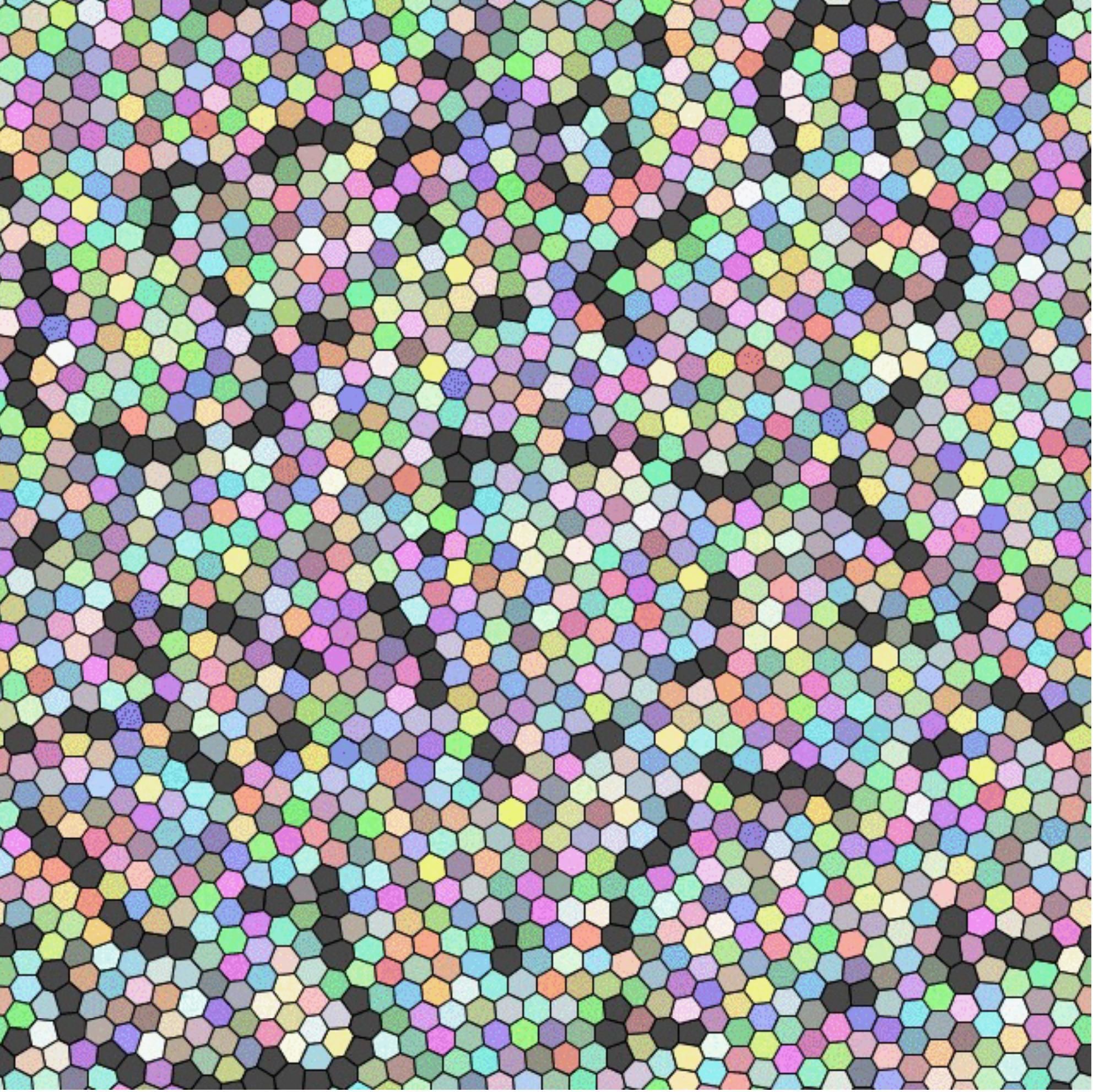}
\caption{720 points after 157 iterations}
\label{F3_emanuele}
\end{figure}
\medskip


\section{The contribution of a single Voronoi cell of the deformed lattice}\lb{S2}


From now on, we adopt the setting defined in the previous section, and we begin with some elementary geometrical observations used in computing the continuous functional $\cF$. This is the first step in the proof of Theorem \ref{T-LimG}.

Under the assumptions of Theorem \ref{T-LimG}, our goal is to compute 
$$
\int_{V(X(\eps(k_1e_1+k_2e_2))|\cX_\eps)}|y-X(\eps(k_1e_1+k_2e_2))|^2dy
$$
in terms of the displacement $X(\eps(k_1e_1+k_2e_2))$ and of the centers of the adjacent cells $V(X(\eps(k_1e_1+k_2e_2))|\cX_\eps)$, i.e. :
\medskip
$$
\ba
{}&X(\eps((k_1+1)e_1+k_2e_2))\,,\quad&&X(\eps((k_1-1)e_1+k_2e_2))\,,
\\
&X(\eps((k_1-1)e_1+(k_2+1)e_2))\,,\quad&&X(\eps((k_1+1)e_1+(k_2-1)e_2))\,,
\\
&X(\eps(k_1e_1+(k_2-1)e_2))\,,\quad&&X(\eps(k_1e_1+(k_2+1)e_2))\,,
\ea
$$
\begin{figure}
\center
\includegraphics[width=7cm]{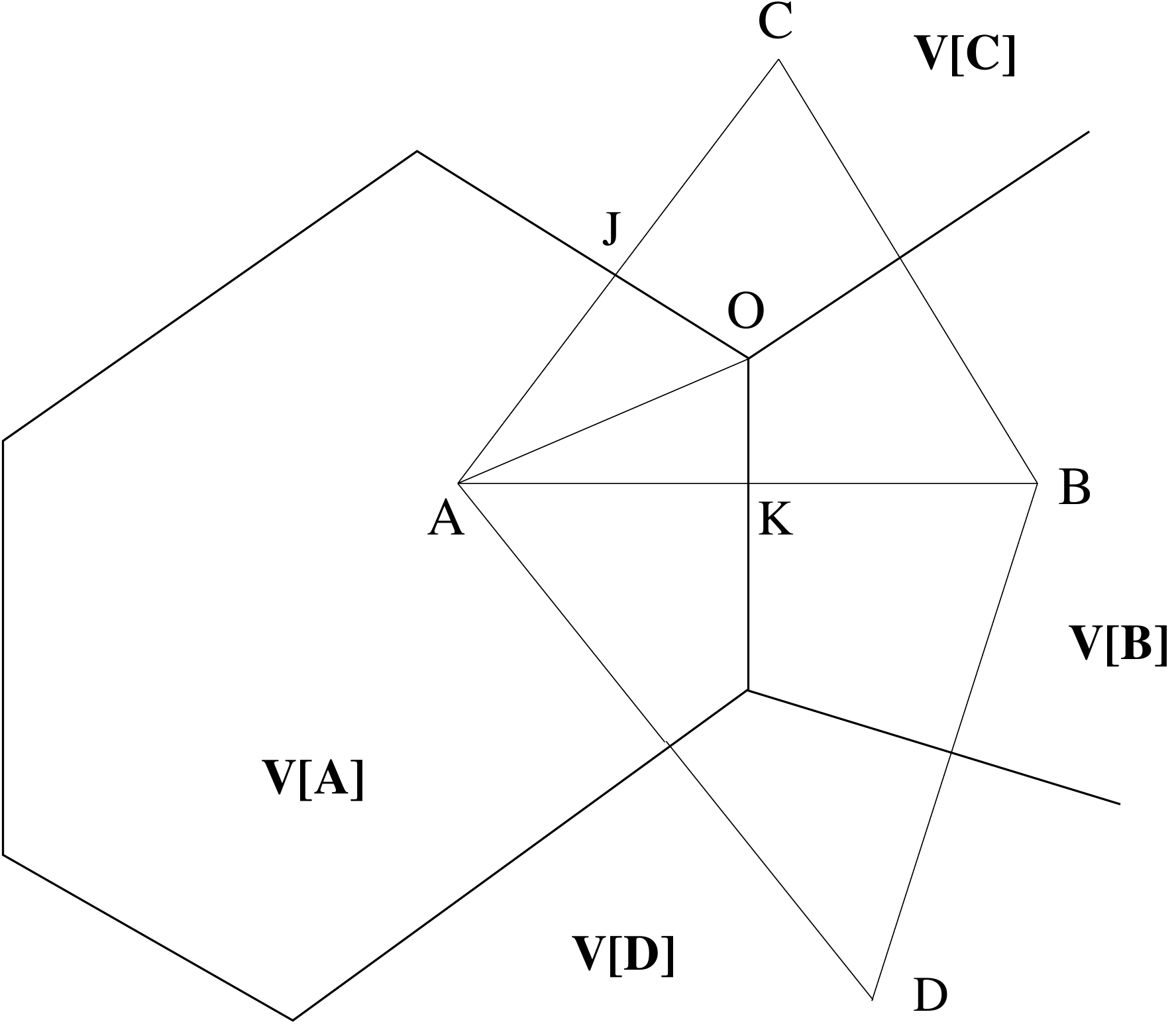}
\caption{Voronoi cell}
\label{F-1}
\end{figure}

To do that, up to sets of measure zero, we can partition each hexagon into $12$ right triangles, each of them similar either to $AOK$ or to $AOJ$ in Figure \ref{F-1}. We start integrating the function $|y-X(\eps(k_1e_1+k_2e_2))|^2$ on one of 
these right triangles. Let $T$ be a right triangle with adjacent sides to the right angle of length $h$ and $l$ (see Figure \ref{F-2}).

\begin{figure}
\center
\includegraphics[width=7cm]{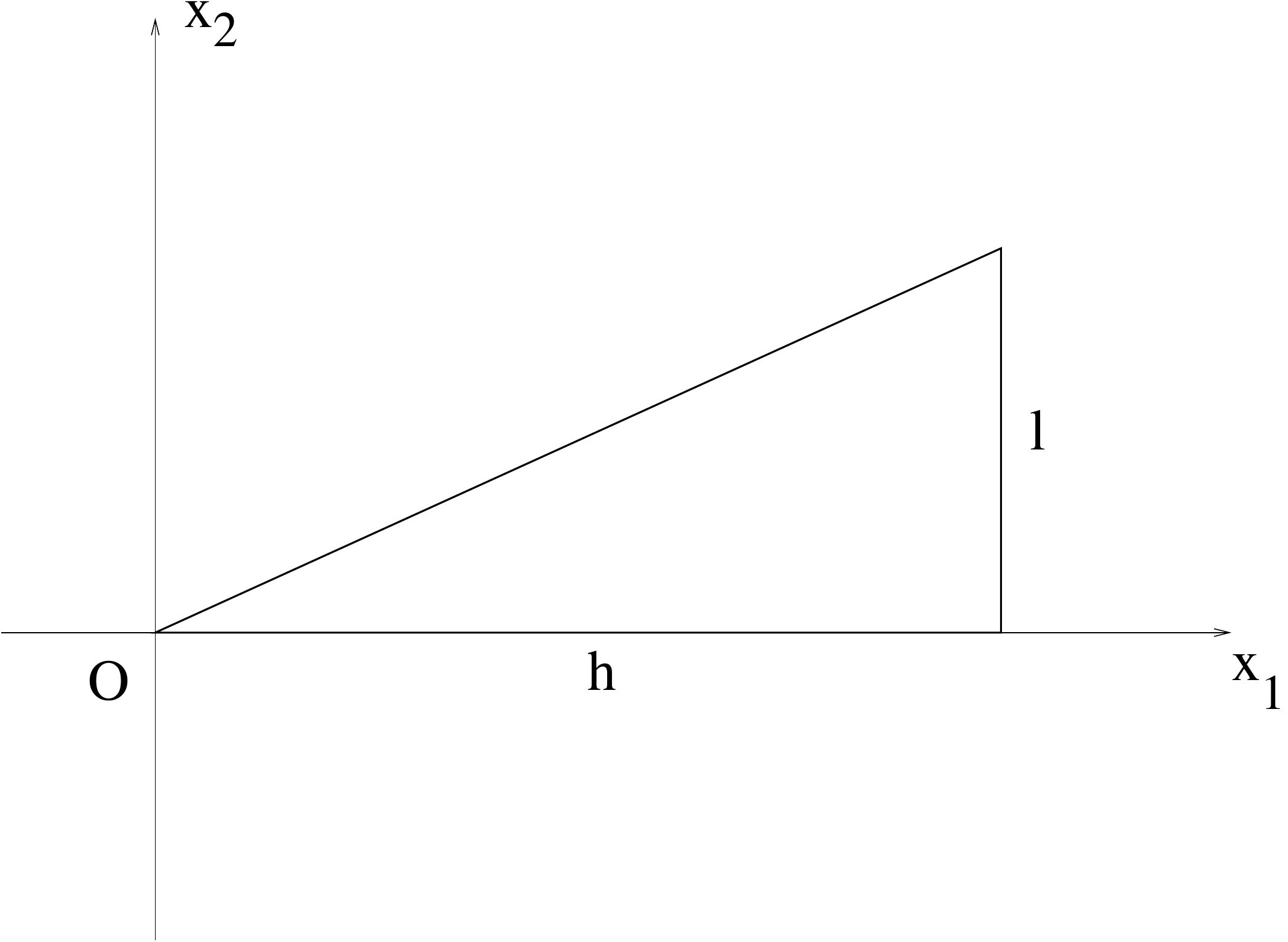}
\caption{Right triangle}
\label{F-2}
\end{figure}

With the notation of Figure \ref{F-2},
$$
\ba
\int_{T}(x_1^2+x_2^2)dx_1dx_2=\int_0^h\int_0^{lx_1/h}(x_1^2+x_2^2)dx_2dx_1=\int_0^h\left(\frac{lx_1}{h}x_1^2+\frac{l^3x_1^3}{3h^3}\right)dx_1
\\
=\int_0^hx_1^3\left(\frac{l}{h}+\frac{l^3}{3h^3}\right)dx_1=\tfrac14lh(h^2+\tfrac13l^2).
\ea
$$
Coming back to the notation of Figure \ref{F-1}, we obtain
$$
\int_{V[A]}|\vec{Ay}|^2dy=\int_{AOK}|\vec{Ay}|^2dy+\int_{AOJ}|\vec{Ay}|^2dy+\hbox{ $10$ similar terms}.
$$
Let us focus on the first term on the right hand side. We compute it in the triangle $ABC$.

\begin{figure}
\center
\includegraphics[width=7cm]{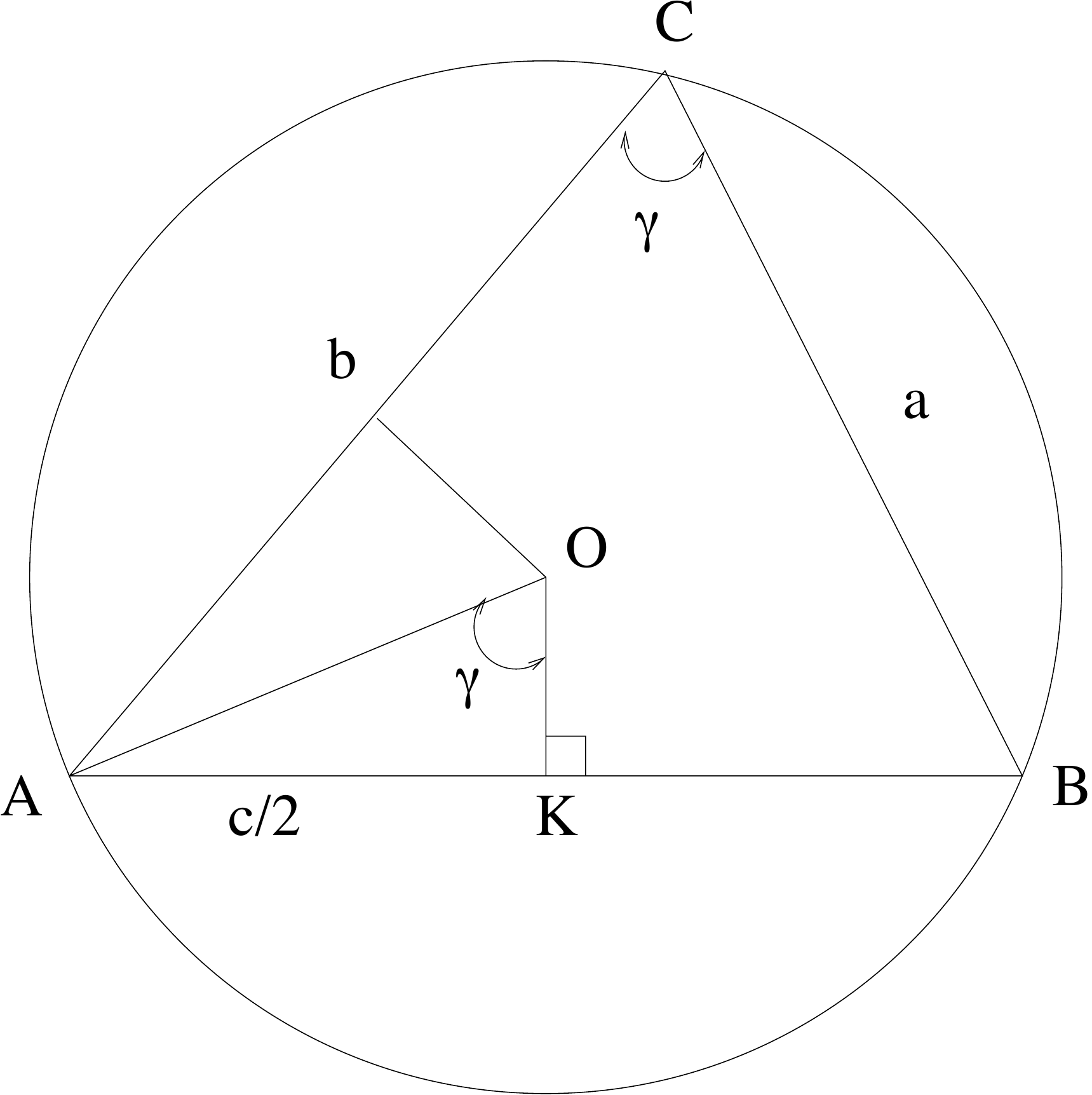}
\caption{Triangle ABC}
\label{F-3}
\end{figure}

Recalling the notation in Figure \ref{F-3}
$$
\int_{AOK}|\vec{Ay}|^2dy=\tfrac18c|OK|(\tfrac14c^2+\tfrac13|OK|^2)\,.
$$
In the triangle $ABC$, we have
$$
\widehat{AOK}=\widehat{ACB}=\g\,,
$$
and
$$
|OK|^2=|OA|^2-\tfrac14c^2=\tfrac14c^2\left(\frac1{\sin^2\g}-1\right).
$$
Moreover, denoting by $S$ the surface of the triangle $ABC$,
$$
2S=|\vec{CA}\wedge\vec{CB}|=ab\sin\g\,.
$$
Thus,
$$
|OK|^2=\tfrac14c^2\left(\frac{a^2b^2}{4S^2}-1\right),
$$
so that
$$
\ba
\int_{AOK}|\vec{Ay}|^2dy=\tfrac1{16}c^2\sqrt{\frac{a^2b^2}{4S^2}-1}\left(\tfrac14c^2+\tfrac1{12}c^2\left(\frac{a^2b^2}{4S^2}-1\right)\right)&
\\
=\tfrac1{192}c^4\sqrt{\frac{a^2b^2}{4S^2}-1}\left(\frac{a^2b^2}{4S^2}+2\right)&\,.
\ea
$$
In other words,
$$
\int_{AOK}|\vec{Ay}|^2dy=\tfrac1{192}|\vec{AB}|^4\sqrt{\frac{|\vec{CB}|^2|\vec{CA}|^2}{|\vec{CA}\wedge\vec{CB}|^2}-1}\left(\frac{|\vec{CB}|^2|\vec{CA}|^2}{|\vec{CA}\wedge\vec{CB}|^2}+2\right)\,.
$$
Exchanging $B$ and $C$, we find by symmetry that
$$
\int_{AOJ}|\vec{Ay}|^2dy=\tfrac1{192}|\vec{AC}|^4\sqrt{\frac{|\vec{BC}|^2|\vec{BA}|^2}{|\vec{BA}\wedge\vec{BC}|^2}-1}\left(\frac{|\vec{BC}|^2|\vec{BA}|^2}{|\vec{BA}\wedge\vec{BC}|^2}+2\right)\,.
$$

Let us now write the latter expression for $A=X(\eps(k_1e_1+k_2e_2))$, where the points $B$ and $C$ are the centers of the Voronoi cells adjacent to the Voronoi cell centered in $A.$ For simplicity of notation, we define
$$
X(\eps(k_1e_1+k_2e_2))=:X_{k_1,k_2}\,.
$$
Thus, the contribution
$$
\int_{V(X(\eps(k_1e_1+k_2e_2))|\cX_\eps)}|y-X(\eps(k_1e_1+k_2e_2))|^2dy
$$
of the terms related to the triangle $ABC$ with $A=X_{k_1,k_2}$, $B=X_{k_1+1,k_2}$ and $C=X_{k_1,k_2+1}$ is:
\be\lb{ContriABC}
\ba
&\int_{AOK}|\vec{Ay}|^2dy+\int_{AOJ}|\vec{Ay}|^2dy
\\
&=\tfrac1{192}|X_{k_1+1,k_2}-X_{k_1,k_2}|^4\sqrt{\frac{|X_{k_1+1,k_2}-X_{k_1,k_2+1}|^2|X_{k_1,k_2}-X_{k_1,k_2+1}|^2}{|(X_{k_1,k_2}-X_{k_1,k_2+1})\wedge(X_{k_1+1,k_2}-X_{k_1,k_2+1})|^2}-1}
\\
&\qquad\qquad\times\left(\frac{|X_{k_1+1,k_2}-X_{k_1,k_2+1}|^2|X_{k_1,k_2}-X_{k_1,k_2+1}|^2}{|(X_{k_1,k_2}-X_{k_1,k_2+1})\wedge(X_{k_1+1,k_2}-X_{k_1,k_2+1})|^2}+2\right)
\\
&+\tfrac1{192}|X_{k_1,k_2+1}-X_{k_1,k_2}|^4\sqrt{\frac{|X_{k_1,k_2+1}-X_{k_1+1,k_2}|^2|X_{k_1,k_2}-X_{k_1+1,k_2}|^2}{|(X_{k_1,k_2}-X_{k_1+1,k_2})\wedge(X_{k_1,k_2+1}-X_{k_1+1,k_2})|^2}-1}
\\
&\qquad\qquad\times\left(\frac{|X_{k_1,k_2+1}-X_{k_1+1,k_2}|^2|X_{k_1,k_2}-X_{k_1+1,k_2}|^2}{|(X_{k_1,k_2}-X_{k_1+1,k_2})\wedge(X_{k_1,k_2+1}-X_{k_1+1,k_2})|^2}+2\right).
\ea
\ee

The total contribution of the integral on the Voronoi cell
$$
\int_{V(X(\eps(k_1e_1+k_2e_2))|\cX_\eps)}|y-X(\eps(k_1e_1+k_2e_2))|^2dy
$$
is therefore the sum of $6$ terms analogous to the right hand side of (\ref{ContriABC}).


\section{The continuous functional $\mc F$}\lb{S3}


In order to derive the formula for the continuous function $\cF$, we need to replace the finite differences appearing on the right hand side of (\ref{ContriABC}) with (partial) derivatives of the deformation, i.e. of the map $X$. This is done by
using Talor's formula, and we arrive at the leading order term in the form:
$$
\ba
&\int_{AOK}|\vec{Ay}|^2dy+\int_{AOJ}|\vec{Ay}|^2dy
\\
&\sim\tfrac1{192}\eps^4|e_1\cdot\grad X(A)|^4\sqrt{\frac{|(e_1-e_2)\cdot\grad X(A)|^2|-e_2\cdot\grad X(A)|^2}{|(-e_2\cdot\grad X(A))\wedge((e_1-e_2)\cdot\grad X(A))|^2}-1}
\\
&\qquad\qquad\times\left(\frac{|(e_1-e_2)\cdot\grad X(A)|^2|-e_2\cdot\grad X(A)|^2}{|(-e_2\cdot\grad X(A))\wedge((e_1-e_2)\cdot\grad X(A))|^2}+2\right)
\\
&+\tfrac1{192}\eps^4|e_2\cdot\grad X(A)|^4\sqrt{\frac{|((e_1-e_2)\cdot\grad X(A))|^2|-e_1\cdot\grad X(A)|^2}{|(-e_1\cdot\grad X(A))\wedge((e_2-e_1)\cdot\grad X(A))|^2}-1}
\\
&\qquad\qquad\times\left(\frac{|((e_1-e_2)\cdot\grad X(A))|^2|-e_1\cdot\grad X(A)|^2}{|(-e_1\cdot\grad X(A))\wedge((e_2-e_1)\cdot\grad X(A))|^2}+2\right).
\ea
$$
It can be simplified as
$$
\ba
&\int_{AOK}|\vec{Ay}|^2dy+\int_{AOJ}|\vec{Ay}|^2dy
\\
&\sim\tfrac1{192}\eps^4|e_1\cdot\grad X(A)|^4\sqrt{\frac{|(e_1-e_2)\cdot\grad X(A)|^2|e_2\cdot\grad X(A)|^2}{|(e_1\cdot\grad X(A))\wedge(e_2\cdot\grad X(A))|^2}-1}
\\
&\qquad\qquad\times\left(\frac{|(e_1-e_2)\cdot\grad X(A)|^2|e_2\cdot\grad X(A)|^2}{|(e_1\cdot\grad X(A))\wedge(e_2\cdot\grad X(A))|^2}+2\right)
\\
&+\tfrac1{192}\eps^4|e_2\cdot\grad X(A)|^4\sqrt{\frac{|(e_1-e_2)\cdot\grad X(A)|^2|e_1\cdot\grad X(A)|^2}{|(e_2\cdot\grad X(A))\wedge(e_1\cdot\grad X(A))|^2}-1}
\\
&\qquad\qquad\times\left(\frac{|(e_1-e_2)\cdot\grad X(A)|^2|e_1\cdot\grad X(A)|^2}{|(e_2\cdot\grad X(A))\wedge(e_1\cdot\grad X(A))|^2}+2\right),
\ea
$$
as $\eps\to 0$. This can be recast as follows:
$$
\ba
&\int_{AOK}|\vec{Ay}|^2dy+\int_{AOJ}|\vec{Ay}|^2dy
\\
&\sim\tfrac1{192}\eps^4|e_1\cdot\grad X(A)|^4\sqrt{\frac{|(e_1-e_2)\cdot\grad X(A)|^2|e_2\cdot\grad X(A)|^2}{\frac{3}{4}|JX(A)|^2}-1}
\\
&\qquad\qquad\times\left(\frac{|(e_1-e_2)\cdot\grad X(A)|^2|e_2\cdot\grad X(A)|^2}{\frac{3}{4}|JX(A)|^2}+2\right)
\\
&+\tfrac1{192}\eps^4|e_2\cdot\grad X(A)|^4\sqrt{\frac{|(e_1-e_2)\cdot\grad X(A)|^2|e_1\cdot\grad X(A)|^2}{\frac{3}{4}|JX(A)|^2}-1}
\\
&\qquad\qquad\times\left(\frac{|(e_1-e_2)\cdot\grad X(A)|^2|e_1\cdot\grad X(A)|^2}{\frac{3}{4}|JX(A)|^2}+2\right)
\ea
$$
as $\eps\to 0$, with the notation $JX:=\Det(\grad X)$.

The total contribution of the Voronoi cell centered in $A$ is the sum of the latter term, plus $5$ analogous contributions obtained by transforming the $3$ unit vectors $e_{12}:=e_1-e_2$, $e_1$, $e_2$ in their images under the action 
of the cyclic group generated by the rotation of $\tfrac{\pi}3.$ 

Since each term is invariant by the symmetry centered in $A,$ we find that:
$$
\ba
&\int_{V(X(\eps(k_1e_1+k_2e_2))|\cX_\eps)}|y-X(\eps(k_1e_1+k_2e_2))|^2dy
\\
&\sim\frac{\eps^4}{48}|e_1\cdot\grad X|^4\Phi(e_1,\grad X)(3+\Phi(e_1,\grad X)^2)(\eps(k_1e_1+k_2e_2))
\\
&+\frac{\eps^4}{48}|e_2\cdot\grad X|^4\Phi(e_2,\grad X)(3+\Phi(e_1,\grad X)^2)(\eps(k_1e_1+k_2e_2))
\\
&+\frac{\eps^4}{48}|e_{12}\cdot\grad X|^4\Phi(e_{12},\grad X)(3+\Phi(e_1,\grad X)^2)(\eps(k_1e_1+k_2e_2))
\ea
$$
as $\eps\to 0$, where
\be\lb{FlaPhi}
\Phi(e,M):=\sqrt{\frac{|MRe|^2|MR^Te|^2}{\frac{3}{4}|\Det(M)|^2}-1}\,,
\ee
and where $R$ is the rotation of an angle $\tfrac{\pi}3$. 

Then the computation above can be summarized as follows:
$$
\int_{V(X(\eps(k_1e_1+k_2e_2))|\cX_\eps)}|y-X(\eps(k_1e_1+k_2e_2))|^2dy\sim\eps^4F(\grad X(\eps(k_1e_1+k_2e_2)))
$$
as $\eps\to 0$, where
\be\lb{FlaF}
F(M):=\tfrac1{48}\sum_{\om\in\{e_1,e_2,e_{12}\}}|M\om|^4\Phi(\om,M)(3+\Phi(\om,M)^2)\,.
\ee

At this point, we recall that $X$ satisfies property (a). Therefore for each $\eps=1/n$ with $n$ a positive integer, the function
$$
y\mapsto\Dist(y,\cX_\eps)\hbox{ is }\scrL\hbox{-periodic.}
$$
Indeed
$$
\Dist(y,\cX_\eps)=y-X(\eps k)\Rightarrow\Dist(y,\cX_\eps)=y+l-X(\eps(k+nl)\ge\Dist(y+l,\cX_\eps)
$$
for each $y\in\bR^2$ and each $l\in\scrL$. Repeating the same argument with $y+l$ and $-l$ instead of $y$ and $l$ shows that
$$
\Dist(y,\cX_\eps)\ge\Dist(y+l,\cX_\eps)\ge\Dist(y,\cX_\eps)
$$
for each $y\in\bR^2$ and $l\in\scrL$, which is precisely the $\scrL-$periodicity condition.

On the other hand, property (a) implies that $X(\Pi)$ is a fundamental domain for $\bR^2/\scrL$. Indeed
$$
X(\Pi+l)=X(\Pi)+l\,,\qquad l\in\scrL
$$
so that
$$
\bigcup_{l\in\scrL}X(\Pi)+l=X\left(\bigcup_{l\in\scrL}\Pi+l\right)=X(\bR^2)=\bR^2\,.
$$
On the other hand
$$
(X(\Pi)+l)\cap X(\Pi)=X(\Pi+l)\cap X(\Pi)=X((\Pi+l)\cap\Pi)\subset X(\d\Pi)
$$
is a set of measure $0$.

Finally
\be\lb{VolXPi}
|X(\Pi)|=1\,.
\ee
By property (a), the deformation map
$$
Y:=X-\id:\,x\mapsto X(x)-x=:Y(x)
$$
is $\scrL$-periodic. We recall the following classical observation.

\begin{lem}\lb{L-DetJac}
Let $Y\in C^1(\bR^2;\bR^2)$ be $\scrL$-periodic. Then
$$
\int_\Pi\Det(\grad Y)(x)dx_1\wedge dx_2=0\,.
$$
\end{lem}

\begin{proof}
If $Y$ is of class $C^2$, one has
$$
\Det(\grad Y)dx_1\wedge dx_2=d(Y_1dY_2)
$$
where $Y_1(x)$ et $Y_2(x)$ are the components of the vector $Y(x)$. Then
$$
\int_\Pi\Det(\grad Y)(x)dx_1\wedge dx_2=\int_\Pi d(Y_1dY_2)=\int_{\d\Pi}Y_1dY_2=0
$$
because $Y$ is $\scrL$-periodic. 

If $Y$ is only of class $C^1$, let $\chi_\a$ be a regularizing sequence. Then, for each $\a>0$, the map $\chi_\a\star Y=:Y_\a\in C^2(\bR^2;\bR^2)$ is still $\scrL$-periodic and $\grad Y_\a\to\grad Y$ uniformly on $\Pi$ as $\a\to 0$. Therefore
$$
\int_\Pi\Det(\grad Y)(x)dx_1\wedge dx_2=\lim_{\a\to 0}\int_\Pi\Det(\grad Y_\a)(x)dx_1\wedge dx_2=0\,.
$$
\end{proof}

\smallskip
Thus
$$
|X(\Pi)|=\int_{X(\Pi)}dy=\int_{\Pi}|\Det(\grad X(x))|dx=\int_{\Pi}\Det(\grad X(x))dx
$$
since $X$ is $C^1$-close to the identity (property (b)), and
$$
\Det(\grad X)=\Det(I+\grad Y)=1+\Div Y+\Det(\grad Y)\,.
$$
Therefore
$$
|X(\Pi)|=\int_\Pi(1+\Div Y+\Det(\grad Y))(x)dx_1\wedge dx_2=1+\int_{\d\Pi}Y\cdot nds=0
$$
by $\scrL$-periodicity of $Y$. This proves (\ref{VolXPi}).

Set
$$
B_R[\scrL]:=\{l\in\scrL\hbox{ s.t. }X(\Pi)+l\subset B(0,R)\}\,,
$$
and $\de:=\Diam(X(\Pi)$. Then
$$
B(0,R-\de)\subset\bigcup_{l\in B_R[\scrL]}(X(\Pi)+l\subset B(0,R))\subset B(0,R)
$$
and since the sets $X(\Pi)+l$ are pairwise disjoint (up to sets of measure $0$) as $l\in\scrL$, we conclude from (\ref{VolXPi}) that
$$
\pi(R-\de)^2=|B(0,R-\de)|\le\#B_R[\scrL]\le|B(0,R)|=\pi R^2
$$
so that
$$
\#B_R[\scrL]\sim\pi R^2\quad\hbox{ as }R\to\infty\,.
$$
Thus
$$
\ba
\#B_R[\scrL]\int_{X(\Pi)}\Dist(y,\cX_\eps)^2dy&\le\int_{B(0,R)}\Dist(y,\cX_\eps)^2dy
\\
&\le\#B_{R+\de}[\scrL]\int_{X(\Pi)}\Dist(y,\cX_\eps)^2dy\,,
\ea
$$
and hence
$$
\frac1{\pi R^2}\int_{B(0,R)}\Dist(y,\cX_\eps)^2dy\to\int_{X(\Pi)}\Dist(y,\cX_\eps)^2dy\hbox{ as }R\to\infty\,.
$$

Thus
$$
\ba
\lim_{R\to\infty}\frac1{R^2}\int_{B(0,R)}\Dist(y,\cX_\eps)^2dy=\int_{X(\Pi)}\Dist(y,\cX_\eps)^2dy
\\
=
\sum_{\eps\max(|k_1|,|k_2|)<1/2}\int_{V(X(\eps(k_1e_1+k_2e_2))|\cX_\eps)}|y-X(\eps(k_1e_1+k_2e_2))|^2dy
\\
\sim\eps^2\int_\Pi F(\grad X(x))dx
\ea
$$
as $\eps\to 0^+$.


\section{Gradient of the functional $\mc F$}


Henceforth we denote
$$
\fH_0:=\left\{Z\in L^2(\bR^2/\scrL;\bR^2)\hbox{ s.t. }Z\hbox{  is }\scrL\hbox{-periodic and}\int_\Pi Z(x)dx=0\right\}
$$
and
$$
\id+\fH_0:=\{\id+Z\hbox{ s.t. }Z\in\fH_0\}\,.
$$
The function $\cF$ is obviously defined on $(\id+\fH_0)\cap\Diff^1(\bR^2)$, and we seek to compute its $L^2$-gradient at $X\in\Diff^2(\bR^2)$.

Let $X\in(\id+\fH_0)\cap\Diff^2(\bR^2)$ and $Y\in\fH_0\cap C^1(\bR^2/\scrL;\bR^2)$; by the implicit function theorem $X+\tau Y\in(\id+\fH_0)\cap\Diff^2(\bR^2)$ for all $\tau$ sufficiently small. With the expression for $\cF$ obtained in 
Theorem \ref{T-LimG}, one anticipates that
\be\lb{GalGradF}
\ba
\frac{d}{d\tau}\cF(X+\tau Y)\rstr_{t=0}=&\int_\Pi\grad F(\grad_xX(x))\cdot\grad_xY(x)dx
\\
=&-\int_\Pi\Div_x(\grad F(\grad_xX(x)))\cdot Y(x)dx\,.
\ea
\ee
In order to verify the second equality above, one only needs to check that $F$ is of class $C^2$ on the set of invertible matrices. Notice indeed that the boundary term coming from Green's formula satisfies
$$
\int_\Pi\grad F(\grad_xX(x)))\cdot Y(x)\otimes n_xds(x)=0
$$
because $\grad_xX(x)=I+\grad_xZ$ and $Y$ are both $\scrL$-periodic. That $F$ is of class $C^2$ on a neighborhood of $I$ in $GL_2(\bR)$ follows from (\ref{FlaF}). Indeed, $\Det(M)\not=0$ for all $M\in GL_2(\bR)$, and one has
$$
\Phi(e,M)^2\to\Phi(e,I)=\tfrac13\quad\hbox{ for all unit vector }e\hbox{ as }M\to I\,.
$$
By continuity, there exists an open neighborhood $\Om$ of $I$ in $GL_2(\bR)$ such that 
$$
\Phi(e,M)^2>\tfrac14\quad\hbox{ for all }(e,M)\in\bS^2\times\Om\,.
$$
Therefore $\Phi(e,\cdot)$ is of class $C^2$ on $\Om$ for each unit vector $e$, and therefore $F$ is of class $C^2$ on $\Om$.

\smallskip
Next we compute $\grad F$. The first step is to compute the directional derivative of $M\mapsto\Phi(e,M)$ at the point $M\in GL_2(\bR)$ along the direction $N \in M_2(\bR)$. We find that
\begin{multline*}
\frac{d}{d\tau}\Phi(e, M+\tau N)|_{\tau=0}
\\
=\frac{1}{2 \Phi(e, M)}\left[\frac{2(NRe|MRe)|MR^Te|^2}{\frac{3}{4}\Det(M)^2}+\frac{2(NR^Te|MR^Te)|MRe|^2}{\frac{3}{4}\Det(M)^2}\right.
\\
\left.-\frac{2|MRe|^2|MR^Te|^2 \Det(M)\Tr(M^{-1}N)}{\frac{3}{4}\Det(M)^3} \right].
\end{multline*}
Recalling that
$$
(Nu|v)=\Tr\big(N(u\otimes v)\big), \qquad u,v \in \bR^2,\quad N\in M_2(\bR),
$$
and 
$$
\frac{d}{d\tau}\Det(M+\tau N)|_{\tau=0}=\Det(M)\Tr(M^{-1}N)\,,\quad M\in GL_2(\bR),\,\,N\in M_2(\bR),
$$
we obtain
\begin{multline*}
\frac{d}{d\tau}\Phi(e, M+\tau N)|_{\tau=0}=\frac{4}{3\Phi(e, M)\Det(M)^2}\left[\Tr[N(Re \otimes MRe)]|MR^Te|^2\right.
\\
\left.+\Tr[N(R^Te\otimes MR^Te)]|MRe|^2-|MRe|^2|MR^Te|^2\Tr(M^{-1}N)\right].
\end{multline*}
Defining
\be\lb{FlaA}
A(e,M):=\frac{(Re\otimes MRe)}{|MRe|^2}+\frac{(R^Te\otimes MR^Te)}{|MR^Te|^2}-M^{-1}\,,
\ee
we see that the map $M\mapsto A(e,M)$ is a tensor-field on $GL_2(\bR)$, homogeneous of degree $-1$ with respect to $M$. The differential of $M\mapsto\Phi(e,M)$ can be easily expressed in terms of $A:$ multiplying and dividing 
each term of the latter equality by $\Phi(e,M)^2+1$, we get
$$
d_M\Phi(e, M)[N]=\frac{\Phi(e,M)^2+1}{\Phi(e, M)}\Tr(A(e,M)N).
$$
In other words, considering the Frobenius inner product defined on $M_2(\bR)$ by 
$$
(M_1|M_2)=\Tr(M_1^TM_2)\,,
$$
the gradient of the map $M\mapsto\Phi(e,M)$ at the point $M\in GL_2(\bR)$ is
$$
\grad_M\Phi(e,M)=\frac{\Phi(e,M)^2+1}{\Phi(e, M)}A(e,M)^T.
$$

Therefore
$$
\ba
\grad F(M)=4\sum_{\om\in\{e_1,e_2,e_{12}\}}|M\om|^2\Phi(\om,M)(3+\Phi(\om,M)^2)\om\otimes(M\om)&
\\
+3\sum_{\om\in\{e_1,e_2,e_{12}\}}|M\om|^4\frac{(1+\Phi(\om,M)^2)^2}{\Phi(\om,M)}A(\om,M)^T&\,.
\ea
$$
Inserting this expression in (\ref{GalGradF}), we find that
$$
\ba
\frac{d}{d\tau}\cF(X+\tau Y)|_{\tau=0}&
\\
=-\tfrac{1}{12}\sum_{\om\in\{e_1,e_2,e_{12}\}}\int_{\Pi}\Div_x(|\om\cdot\grad X|^2(3+\Phi^2)\Phi(\om,\grad X)(\om\cdot\grad X)\otimes\om)\cdot Ydx&
\\
-\tfrac{1}{16}\sum_{\om\in\{e_1,e_2,e_{12}\}}\int_{\Pi}\Div_x\left(|\om\cdot\grad X|^4\frac{(1+\Phi(\om,\grad X)^2)^2}{\Phi(\om,\grad X)}A(\om,\grad X)\right)\cdot Ydx&\,.
\ea
$$

In other words, the $L^2$-gradient of $\cF$ is given by the formula
\be\lb{FlaGradF}
\ba
\frac{\de\cF(X)}{\de X(x)}=&-\Div_x(\grad F(\grad_xX(x))
\\
=&-\tfrac{1}{12}\sum_{\om\in\{e_1,e_2,e_{12}\}}\Div_x(|\om\cdot\grad X|^2(3+\Phi^2)\Phi(\om,\grad X)\om\otimes(\om\cdot\grad X))
\\
&-\tfrac{1}{16}\sum_{\om\in\{e_1,e_2,e_{12}\}}\Div_x\left(|\om\cdot\grad X|^4\frac{(1+\Phi(\om,\grad X)^2)^2}{\Phi(\om,\grad X)}A(\om,\grad X)^T\right)\,.
\ea
\ee
Its $i$-th coordinate is given by the expression 
\be\lb{FlaGradFi}
\ba
\frac{\de\cF(X)}{\de X_i(x)}=&-\tfrac{1}{12}\sum_{\om\in\{e_1,e_2,e_{12}\}}\d_j(|\om\cdot\grad X|^2(3+\Phi^2)\Phi(\om,\grad X)\om_j\om_k\d_kX_i))
\\
&-\tfrac{1}{16}\sum_{\om\in\{e_1,e_2,e_{12}\}}\d_j\left(|\om\cdot\grad X|^4\frac{(1+\Phi(\om,\grad X)^2)^2}{\Phi(\om,\grad X)}A_{ji}(\om,\grad X)\right)\,,
\ea
\ee
with the usual convention of summation on repeated indices.


\section{The asymptotic energy functional \\ for a slightly deformed hexagonal lattice}\lb{S5}


In this section we study the functional $\cF(X)$ near $X=\id$. More precisely, we seek a rational expression of $F(M)$ (with $F$ defined by (\ref{FlaF}) for $M\in GL_2(\bR)$ near $I$. In view of Theorem \ref{T-LimG}, this gives a simplified
expression of $\cF(X)$ near $X=\id$.

Let
$$
M_1=\left(
\begin{matrix}
 \a & \b \\
 \g & \de \\
\end{matrix}
\right)
$$
be a near-identity matrix. Straightforward computations show that
\be\lb{FlaPhi1}
\Phi(e_1,M_1)=\tfrac1{2 \sqrt{3}}\sqrt{\frac{\left(\a^2-3 \b^2+\g^2-3 \de^2\right)^2}{(\b \g-\a \de)^2}}=\frac{-\a^2+3 \b^2-\g^2+3 \de^2}{2 \sqrt{3}(\a\de-\b\g)}\,.
\ee
Indeed, since $|M_1-I|\ll 1$, one has $\a^2-3 \b^2+\g^2-3 \de^2<0$.

Next, we observe that
$$
\ba
\Phi(Re,M)=&\sqrt{\frac{|MR^2e|^2|Me|^2}{\frac{3}{4}|\Det(M)|^2}-1}
\\
=&\sqrt{\frac{|R^TMR^2e|^2|R^TMRR^Te|^2}{\frac{3}{4}|\Det(R^TMR)|^2}-1}=\Phi(e,R^TMR)
\ea
$$
for each $M\in GL_2(\bR)$ and each unit vector $e$. Indeed, the second equality above follows from the obvious identity $|R\xi|=|R^T\xi|=|\xi|$ for all $\xi\in\bR^2$ since $R$ is a rotation. Therefore
\be\lb{FlaPhiRot}
\ba
{}&\Phi(e_2,M_1)\,=\Phi(e_1,M_2)&&\quad\hbox{ with }M_2\,=R^TM_1R
\\
&\Phi(e_{12},M_1)\!=\Phi(e_1,M_{12})&&\quad\hbox{ with }M_{12}\!=RM_1R^T
\ea
\ee
since $e_2=Re_1$ and $e_{12}=R^Te_1$.

Elementary computations show that
$$
M_2=\tfrac14\left(\begin{matrix}
\a+\sqrt{3}\g+\sqrt{3}\b+3\de &-\sqrt{3}\a-3\g+\b+\sqrt{3}\de
\\	\\
-\sqrt{3}\a+\g-3\b+\sqrt{3}\de &3\a-\sqrt{3}\g-\sqrt{3}\b+\de
\end{matrix}\right)
$$
and
$$
M_{12}=\tfrac14\left(\begin{matrix}
\a-\sqrt{3}\g-\sqrt{3}\b+3\de &\sqrt{3}\a-3\g+\b-\sqrt{3}\de
\\	\\
\sqrt{3}\a+\g-3\b-\sqrt{3}\de &3\a+\sqrt{3}\g+\sqrt{3}\b+\de
\end{matrix}\right)
$$
Notice that 
$$
|M_1-I|\ll 1\Rightarrow |M_2-I|=|R^T(M_1-I)R|\ll 1\hbox{ and }|M_{12}-I|=|R(M_1-I)R^T|\ll 1.
$$
Hence we can use formula (\ref{FlaPhi1}) to compute $\Phi(e_2,M_1)$ and $\Phi(e_{12},M_1)$ with the help of  and (\ref{FlaPhiRot}). We find that
\be\lb{FlaPhi2}
\Phi(M_1,e_2)=\frac{\sqrt{3} \a^2-3 \a \b+\sqrt{3} \g^2-3 \g \de}{3 ( \a \de - \b \g )}\,,
\ee
and 
\be\lb{FlaPhi12}
\Phi(M_1,e_{12})=\frac{{\sqrt{3}}{\a^2}+3\a \b+{\sqrt{3}}{\g^2}+3\g \de}{3(\a \de - \b \g)}\,.
\ee

Finally, we insert the expressions found in (\ref{FlaPhi1}), (\ref{FlaPhi2}) and (\ref{FlaPhi12}) in formula (\ref{FlaF}), and find that
\be
F(M_1)=\tfrac1{96\sqrt{3}}\frac{P(\a,\b,\g,\de)}{\a \de - \b \g}
\ee
where
\be
\ba
P(\a,\b,\g,\de)=&-\a^6+6 \a^4 \b^2-9 \a^2 \b^4-3 \a^4 \g^2
\\
&+18 \a^2 \b^2 \g^2+9 \b^4 \g^2-3 \a^2 \g^4+12 \b^2 \g^4-\g^6-12 \a^3 \b \g \de
\\
&-36 \a \b^3 \g \de-12 \a \b \g^3 \de+12 \a^4 \de^2+18\a^2 \g^2 \de^2
\\
&+6 \g^4 \de^2-36 \a \b \g \de^3+9 \a^2 \de^4-9 \g^2 \de^4\,.
\ea
\ee
We shall simplify this expression, and more precisely give an intrinsic formula for the polynomial $P$. Set
$$
S=\left(
\begin{matrix}
 1 & 0 \\
 0 & -1 \\
\end{matrix}
\right)\,.
$$
Elementary (although tedious) computations show that
$$
P(\a,\b,\g,\de)=\tfrac12(Q_+(\a,\b,\g,\de)+Q_-(\a,\b,\g,\de))
$$
with
$$
\ba
Q_+(\a,\b,\g,\de)&:=(\a^2+\b^2+\g^2+\de^2)\left(24(\a\de-\b\g)^2-(\a^2+\b^2+\g^2+\de^2)^2\right)\,,
\\
Q_-(\a,\b,\g,\de)&:=(\a^2-\b^2+\g^2-\de^2)\left(12(\a\b+\g\de)^2-(\a^2-\b^2+\g^2-\de^2)^2\right)\,.
\ea
$$
One has
$$
\ba
(\a^2+\b^2+\g^2+\de^2)&=\Tr(M_1^TM_1)\,,
\\
(\a^2-\b^2+\g^2-\de^2)&=\Tr(M_1^TM_1S)\,,
\ea
$$
while
$$
\ba
(\a\b+\g\de)^2=&(\a^2+\g^2)(\b^2+\de^2)-(\a\de-\b\g)^2
\\
=&\Tr(M_1^TM_1\tfrac{I+S}2)\Tr(M_1^TM_1\tfrac{I-S}2)-\Det(M_1)^2
\\
=&\tfrac14\left(\Tr(M_1^TM_1)^2-\Tr(M_1^TM_1S)^2\right)-\Det(M_1)^2\,.
\ea
$$
Therefore
$$
\ba
P(\a,\b,\g,\de)=&6\Det(M_1)^2\Tr(M_1^TM_1(2S-I))
\\
&+\tfrac32\Tr(M_1^TM_1)^2\Tr(M_1^TM_1S)
\\
&-\tfrac12\Tr(M_1^TM_1)^3-2\Tr(M_1^TM_1S)^3\,.
\ea
$$
Hence
$$
\ba
F(M_1)=&\tfrac1{16\sqrt{3}}\Det(M_1)\Tr(M_1^TM_1(2S-I))
\\
&+\tfrac1{64\sqrt{3}}\frac{\Tr(M_1^TM_1)^2\Tr(M_1^TM_1S)}{\Det(M_1)}
\\
&-\tfrac1{192\sqrt{3}}\frac{\Tr(M_1^TM_1)^3}{\Det(M_1)}-\tfrac1{48\sqrt{3}}\frac{\Tr(M_1^TM_1S)^3}{\Det(M_1)}\,.
\ea
$$

We finally compute the Taylor expansion of $48F$ at order $3$ near the identity matrix. Setting 
$$
N=\left(
\begin{matrix}
 a & b \\
 c & d
\end{matrix}
\right)
$$
we find that
\be\lb{eq:Taylor}
\ba
48F(I+\eps N)=\tfrac{10}{\sqrt{3}}+\tfrac{20}{\sqrt{3}}\eps(a+d)
\\
+\tfrac1{\sqrt{3}}\eps^2(13 a^2+3 b^2-14 b c+3 c^2+34 ad+13 d^2)
\\
+\tfrac1{\sqrt{3}}\eps^3(a^3+9ab^2-2 abc+9 ac^2+23 a^2 d-3 b^2 d-26 bcd-3 c^2 d+11 ad^2+5 d^3)
\\
+O(\eps^4)
\ea
\ee
It is interesting to notice that the Taylor expansion of $F$ around the identity matrix is invariant under the substitutions $a\leftrightarrow d$ and $b \leftrightarrow c$ only up to second order. More precisely
\be\lb{eq:TaylorIntr}
\ba
48F(I+\eps N)=\tfrac{10}{\sqrt{3}}+\tfrac{20}{\sqrt{3}}\eps\Tr(N)
\\
+\tfrac1{\sqrt{3}}\eps^2\left(14\Det(N)+10\Tr(N)^2+3\Tr(N^TN)\right)&+O(\eps^3)\,.
\ea
\ee


\section{Stability and asymptotic convergence for small perturbations}\lb{S6}


In this section we use a perturbative approach to study stability properties of the energy functional $\cF(X)$ around the identity. 

Following formula (\ref{GalGradF}), we consider the PDE defining the gradient flow of $\cF$ in the form
\be\lb{CauchyPbGFlow}
\d_tX(t,x)=\Div_x(\grad F(\grad_xX(t,x))\,,\qquad X(0,x)=X^{in}(x)\,.
\ee
We assume that $X^{in}$ satisfies properties (a-c), and we seek a (weak) solution of the Cauchy problem (\ref{CauchyPbGFlow}) such that $X(t,\cdot)\in\Diff^1(\bR^2)$ satisfies (properties (a-c) for all $t\ge 0$. In particular, property (c) 
is preserved by the evolution of (\ref{CauchyPbGFlow}) since the system of PDEs governing $X$ is in divergence form.

Therefore, we henceforth seek $X$ of the form
$$
X(t,x)=x+\eps Y(t,x)
$$
with $0<\eps\ll 1$, and property (a) implies that $Y(t,\cdot)$ is a $\scrL$-periodic map from $\bR^2$ to itself.

\smallskip
\noindent
{\it Step 1: Convexification of the problem.}

Define the function $F_0$ on $GL_2(\bR)$ as follows:
\be\lb{DefF0}
F_0(M):=F(M)-\tfrac{5}{12\sqrt{3}}\Tr(M-I)-\tfrac{7}{24\sqrt{3}}\Det(M-I)\,.
\ee
Thus
$$
F(I+\eps\grad_xY)=F_0(I+\eps\grad_xY)+\tfrac{5}{12\sqrt{3}}\eps\Div_x(Y)+\tfrac{7}{24\sqrt{3}}\eps^2\Det(\grad_xY)\,.
$$
Therefore
$$
\cF(\id+\eps Y)=\int_\Pi F(I+\eps\grad_xY(x))dx=\int_\Pi F_0(I+\eps\grad_xY(x))dx
$$
since
$$
\int_\Pi\Div_x(Y)(x)dx=0\quad\hbox{ and }\int_\Pi\Det(\grad_xY(x))dx=0\,.
$$
The first equality is obvious since $Y$ is $\scrL$-periodic, while the second follows from Lemma \ref{L-DetJac}.

With formula (\ref{eq:TaylorIntr}), we see that
\be\lb{F0NearI}
F_0(I+\eps N)=\tfrac{5}{24\sqrt{3}}+\tfrac5{24\sqrt{3}}\eps^2\Tr(N)^2+\tfrac1{16\sqrt{3}}\eps^2\Tr(N^TN)+O(\eps^3)\,.
\ee
Formula (\ref{DefF0}) shows that $F_0\in C^\infty(GL_2(\bR))$ since $F\in C^\infty(GL_2(\bR))$. Then, formula (\ref{F0NearI}) implies that 
\be\lb{D2F0I}
\grad^2F_0(I)\cdot(N,N)\ge\tfrac1{8\sqrt{3}}\Tr(N^TN)\,.
\ee
Henceforth we denote by $\|\cdot\|_2$ the Frobenius norm on $M_2(\bR)$, defined by the formula
$$
\|A\|_2=\Tr(A^TA)^{1/2}\,.
$$
The inequality (\ref{D2F0I}) implies that there exists three positive constants $0<\l\le\L$ and $\rho_0$ such that
\be\lb{CoercF0}
\l\|N\|_2^2\le\grad^2F_0(A)\cdot(N,N)\le\L\|N\|_2^2\,,\quad\hbox{ for all }A\hbox{ such that }\|A-I\|_2<\rho_0
\ee

Choose $G\in C^2(M_2(\bR))$ such that
$$
\|A-I\|_2<\rho_0/2\Rightarrow G(A)=F_0(A)
$$
while
$$
\tfrac12\l\|N\|_2^2\grad^2G(A)\cdot(N,N)\le 2\L\|N\|_2^2\quad\hbox{ for all }A,N\in M_2(\bR)\,.
$$
Instead of (\ref{CauchyPbGFlow}), consider the Cauchy problem
\be\lb{CauchyPbGFlowG}
\d_tX(t,x)=\Div_x(\grad G(\grad_xX(t,x))\,,\qquad X(0,x)=X^{in}(x)\,.
\ee
Let $X$ be the solution of this Cauchy problem.

\smallskip
\noindent
{\it Step 2: Stability in $L^2$.}

Multiplying both sides of (\ref{CauchyPbGFlow}) by $X(t,x)-x$ and integrating over $\Pi$, one finds that
$$
\frac{d}{dt}\tfrac12\int_\Pi|X(t,x)-x|^2dx=-\int_\Pi\Tr(\grad G(\grad_xX(t,x))^T(\grad_xX(t,x)-I))dx\,.
$$
Since $\grad G(I)=0$, for each $M\in M_2(\bR)$, one has
$$
\ba
\grad G(M)=&\int_0^1\Tr((\grad G(I+s(M-I)-\grad G(I))^T(M-I))ds
\\
\ge&\int_0^1\int_0^s\grad^2G(I+s(M-I)\cdot(M-I,M-I))ds
\\
\ge&\tfrac14\l\|M-I\|_2^2\,.
\ea
$$
Hence
$$
\frac{d}{dt}\tfrac12\int_\Pi|X(t,x)-x|^2dx\le-\tfrac14\l\int_\Pi\|\grad_xX(t,x)-I\|_2^2dx\,.
$$
Since $x\mapsto X(t,x)-x$ is $\scrL$-periodic by property (a), we deduce from the Poincar\'e-Wirtinger inequality that 
$$
\int_\Pi|X(t,x)-x|^2dx\le C_P\int_\Pi\|\grad_xX(t,x)-I\|_2^2dx
$$
(denoting by $C_P$ the best constant in the Poincar\'e-Wirtinger inequality). Therefore
$$
\frac{d}{dt}\tfrac12\int_\Pi|X(t,x)-x|^2dx\le-\tfrac14\l C_P\int_\Pi|X(t,x)-x|^2dx\,,
$$
so that
\be\lb{ExpDecay}
\int_\Pi|X(t,x)-x|^2dx\le e^{-C_P\l t/2}\int_\Pi|X^{in}-x|^2dx\,.
\ee

\smallskip
\noindent
{\it Step 3: Uniform stability.}

Next we prove that the solution $X$ of he Cauchy problem (\ref{CauchyPbGFlowG}) remains close enough to the identity map so that $G$ existence and uniqueness for the gradient flow of $\cF$ by showing that, for initial data sufficiently close to 
the identity, it coincides with the gradient flow of $\cG$.

Assume that
$$
\|X^{in}-\id\|_{W^{\sigma,p}(\Pi)}\le\eps_0
$$
with $p>2$ and $1+2/p<\sigma<2$. By the Theorem on page 192 in \cite{A}, there exists $t_0>0$ such that the solution $X$ of the Cauchy problem (\ref{CauchyPbGFlowG}) satisfies\footnote{Theorem on page 192 in \cite{A} considers solutions 
in bounded domains. However, this result is based on abstract results on evolution equations that apply also to the periodic case.}
$$
\|X(t,\cdot)-\id\|_{W^{\sigma,p}(\Pi)}\le 2\eps_0\quad\hbox{ for all }t\in[0,t_0]\,.
$$
Since $(\sigma-1)p>2$, by Sobolev embedding on the $2$-dimensional torus $\Pi$ one has
\be\lb{NearIdInit}
\|X(t,\cdot)-\id\|_{C^{1,\alpha}(\Pi)}\le C\eps_0\quad\hbox{ for all }t\in[0,t_0]\,,
\ee
for some positive $\alpha\equiv\alpha(\sigma,p)$ and $C\equiv C(\alpha,\sigma,p)$. 

Next consider $(\bar x,\bar t)$ with $\bar t\ge t_0$, together with the parabolic cylinder
$$
Q_{t_0}(\bar x, \bar t):=\{(x,t) \in\Pi\times\bR\hbox{ s.t. }t\in[\bar t-t_0, \bar t]\hbox{ and }|x-\bar x|\le\sqrt{t_0}\}.
$$
Let us now compute 
$$
\int_{Q_{t_0}(\bar x,\bar t)}|X(t,x)-x|^2dxdt\le\int_{\bar t-t_0}^{\bar t}\int_{\Pi}|X(t,x)-x|^2dxdt\,.
$$
By \eqref{ExpDecay}
$$
\ba
\int_{\bar t-t_0}^{\bar t}\int_{\Pi}|X(t,x)-x|^2dxdt&\le\int_{\bar t-t_0}^{\bar t}e^{-C_P\l t}\int_{\Pi}|X^{in}(x)-x|^2dxdt
\\
&\le\frac{1}{C_P\l}\|X^{in}-\id\|_2^2\,.
\ea
$$
In particular,
$$
\frac{1}{|Q_{t_0}(\bar x,\bar t)|}\int_{Q_{t_0}(\bar x,\bar t)}|X(t,x)-x|^2dxdt\le\frac{1}{C_P\pi t_0^2}\|X^{in}-\id\|_2^2\le\frac{\eps_0^2}{C_P\pi t_0^2}\,.
$$
Thus, in the parabolic cylinder $Q_{t_0}(\bar x,\bar t)$, the map $X$ is $L^2$-close to the identity map, and we seek to improve this result into a similar statement with the $C^{1,\a}$ instead of $L^2$ topology. This is done by appealing to the 
local regularity theory of parabolic equations. Specifically, we apply the A-caloric approximation argument in \cite{DM}. Since $t_0$ is fixed and $\eps_0$ can be chosen arbitrarily small, given any point $(\hat x,\hat t)\in Q_{t_0/2}(\bar x,\bar t))$
we can apply \cite[Lemma 7.3]{DM} with $M=2$, $\rho=t_0/2$, and $\ell_\rho=0$, to deduce that there exists a vector $\Ga_{\hat x,\hat t}\in\bR^2$ and a positive constant $c$ that is independent of $\hat x$ and $\hat t$ such that
$$
\frac{1}{|Q_r|}\int_{Q_r(\hat x,\hat t)}|\nabla X(t,x)-\Gamma_{\hat x,\hat t}|^2dxdt\le cr^{2\beta}\,,\qquad\hbox{ for all }r\in(0,t_0/4)\,.
$$
This means that $\nabla X$ belongs to a Campanato space, which is known to coincide with the classical H\"older space \cite{Camp}. Thus
$$
\|\grad X-I\|_{C^{0,\beta}(Q_{t_0/2}(\bar x,\bar t))}\le\bar c\,,
$$
with $\bar c$ independent of $\bar x$ and $\bar t$. 

By localization, interpolation with (\ref{ExpDecay}) and Sobolev embedding, we see that
$$
\|X(t,\cdot)-\id\|_{L^\infty(B(\bar x,\sqrt{t_0}/2)}\le C_S(\th)\bar c^\th e^{-(1-\th)C_P\l t/4}\eps^{1-\th}_0\,,\quad t\in[\bar t-t_0,\bar t]
$$
for all $\th\in(\tfrac23,1)$, where $C_S(\th)$ denotes the Sobolev constant for the embedding $W^{\th,4/(2-\th)}(B(\bar x,\sqrt{t_0}/2)\subset L^\infty(B(\bar x,\sqrt{t_0}/2)$.  (Indeed, applying Theorem 6.4.5 (7) in \cite{BergLof} with $s_0=0$, $p_0=2$, 
$s_1=1$ and $p_1=4$ shows that 
$$
\|X(t,\cdot)-\id\|_{W^{\th,4/(2-\th)}(B(\bar x,\sqrt{t_0}/2)}\le \bar c^{\th} e^{-(1-\th)C_P\l t/4}\eps^{1-\th}_0\,,
$$
and $W^{\th,4/(2-\th)}(B(\bar x,\sqrt{t_0}/2)\subset L^\infty(B(\bar x,\sqrt{t_0}/2)$ provided that $\th>1-\th/2$ by Sobolev's embedding theorem.) 

By a classical argument\footnote{Let $f\in C^{1,\b}(B(0,R))$ for some $\b\in(0,1)$. Then 
$$
\|\grad f\|_{L^\infty(B(0,R)}\le\left(\tfrac2\b\right)^{\b/(\b+1)}\|f\|^{1/(\b+1)}_{C^{1,\b}(B(0,R)}\|f\|^{\b/(\b+1)}_{L^\infty(B(0,R)}\,.
$$
Indeed, by the Mean Value Theorem
$$
f(x+h)-f(x)=\grad f(x)\cdot h+(\grad f(x+sh)-\grad f(x))\cdot h
$$
for some $s\in(0,1)$, so that
$$
|\grad f(x)|\le\frac{2\|f\|_{L^\infty(B(0,R)}}{|h|}+|h|^{1+\b}\|\grad f\|_{C^{0,\b}(B(0,R)}\,.
$$
Optimizing in $|h|$ leads to the conclusion.}
$$
\|\grad X(t,\cdot)-I\|_{L^\infty(B(\bar x,\sqrt{t_0}/2)}\le\left(\tfrac2\b\right)^{\frac{\b}{\b+1}}C_S(\th)^{\frac1{1+\b}}\bar c^{\frac{1+\th\b}{1+\b}} e^{-\frac{(1-\th)\b C_P\l t}{4(1+\b)}}\eps^{\frac{\b(1-\th)}{1+\b}}_0\,,\quad |\bar t-t|\le t_0\,.
$$
With (\ref{NearIdInit}), this implies that, for all $t\ge 0$, one has
$$
\|\grad X(t,\cdot)-I\|_{L^\infty(\Pi)}\le\max\left(C\eps_0,\left(\tfrac2\b\right)^{\frac{\b}{\b+1}}C_S(\th)^{\frac1{1+\b}}\bar c^{\frac{1+\th\b}{1+\b}} e^{-\frac{(1-\th)\b C_P\l t}{4(1+\b)}}\eps^{\frac{\b(1-\th)}{1+\b}}_0\right)\,.
$$

\smallskip
\noindent
{\it Step 4: Conclusion.}

By choosing $\eps_0$ small enough, we conclude that
$$
\|\grad X(t,\cdot)-I\|_{L^\infty(\Pi)}\le\rho_0/2\,,\qquad t\ge 0\,,
$$
so that
$$
\grad G(\grad_xX(t,x))=\grad F_0(\grad_xX(t,x))\,,\quad t\ge 0\,,\,\,x\in\Pi\,.
$$
Hence $X$ satisfies
$$
\d_tX(t,x)=\Div_x(\grad F_0(\grad_xX(t,x))\,,\quad t>0\,,\,\,x\in\Pi\,.
$$
On the other hand, (\ref{DefF0}) implies that
$$
(\grad F(M)-\grad F_0(M))\cdot N=\tfrac{1}{8\sqrt{3}}\Tr(N)+\tfrac{7}{24\sqrt{3}}\Det(M)\Tr(M^{-1}N)
$$
so that
$$
\grad F(\grad_xX(t,x))-\grad F_0(\grad_xX(t,x))=\tfrac1{8\sqrt{3}}I+\tfrac{7}{24\sqrt{3}}JX(t,x)(\grad_xX(t,x)^{-1})^T\,.
$$
Since
$$
JX(\grad_xX^{-1})^T=
\left(\begin{matrix}
\d X_2/\d x_2 & -\d X_2/\d x_1
\\
-\d X_1/\d x_2&\d X_1\d x_1
\end{matrix}\right)
$$
one has $\Div_x(JX(\grad_xX^{-1})^T)=0$, so that
$$
\Div_x\grad F(\grad_xX(t,x)))=\Div_x(\grad F_0(\grad_xX(t,x)))\,.
$$
In other words, $X$ is in fact the solution of (\ref{CauchyPbGFlow}). 

Finally, for each $\th\in(\tfrac23,1)$, one has
$$
|X(t,\bar x)-\id|\le \max(2,C_S(\th)\bar c^\th)e^{-(1-\th)C_P\l(t-t_0)/4}\eps^{1-\th}_0\,,\quad t>0\,,
$$
and the proof is complete.


\bigskip

{\it Acknowledgments:}  The first author is grateful to Riccardo Salvati Manni for useful discussions on section \ref{S5}. The third author is grateful to Giuseppe Mingione for useful suggestions about regularity theory. Moreover, the third author 
would like to acknowledge the L'Or\'eal Foundation for partially supporting this project by awarding the author with the L'Or\'eal-UNESCO \emph{For Women in Science France fellowship}.


\end{document}